\documentclass[12pt]{amsart}
\usepackage{geometry} % see geometry.pdf on how to lay out the page. There's lots.
\usepackage{amssymb}
\usepackage{amscd}
\usepackage{amsmath}
\usepackage[all,cmtip]{xy}
\usepackage{mathrsfs}
\usepackage{amsthm}
\usepackage[pdfpagelabels]{hyperref}
\newfont{\Bb}{msbm10 scaled1200}
\newfont{\bb}{msbm8 scaled1200}

\def\CC{\mathbb{C}}
\def\eop{\hspace*{\fill}$\Box$ \vskip \baselineskip}

\def\OO{\mathcal O}
\def\beq{\begin{equation}}
\def\eeq{\end{equation}}
\def\beqa{\begin{eqnarray}}
\def\eeqa{\end{eqnarray}}

\def\ld{\ldots}
\def\cd{\cdots}
\def\DD{^\circ\hskip -2pt D}
\def\ID{I\hskip -1pt D}

\newtheorem{thm}{Theorem} [section]
\newtheorem*{mainthm}{Main Theorem}%[section]
\newtheorem{cor}[thm]{Corollary}
\newtheorem{lem}[thm]{Lemma}
\newtheorem{prop}[thm]{Proposition}
\theoremstyle{definition}
\newtheorem{defn}[thm]{Definition}

\newtheorem{example}[thm]{Example}

\newtheorem{rem}[thm]{Remark}

\begin{document}

\title{Free resolutions for multiple point spaces}

\author{Ay{\c s}e Alt{\i}nta{\c s}}
\address{Department of Mathematics, Y{\i}ld{\i}z Technical University, Esenler 34210, Istanbul, Turkey}
\email{aysea@yildiz.edu.tr}
\thanks{The first author was supported by a Vice Chancellor's scholarship at the University of Warwick.}

\author{David Mond}
\address{Mathematics Institute, the University of Warwick, Coventry CV4 7AL, UK}
\email{D.M.Q.Mond@warwick.ac.uk}

\subjclass[2000]{58K20}

\date{November 10, 2011}


%\keywords{}

\begin{abstract}
Let $f:(\CC^n,0)\to(\CC^{n+1},0)$ be a map-germ of corank 1, and, for $1\leq k\leq
\text{multiplicity}(f)$, let $D^k(f)$ be its
$k$'th multiple-point scheme -- the closure of the set of ordered $k$-tuples of 
pairwise distinct points sharing the same image. There are natural projections 
$D^{k+1}(f)\to D^k(f)$, determined by forgetting one member of the $k+1$-tuple. 
We prove that the matrix of a presentation of 
$\OO_{D^{k+1}(f)}$ over $\OO_{D^k(f)}$ appears as a certain 
submatrix of the matrix of a suitable presentation of 
$\OO_{\CC^n}$ over $\OO_{\CC^{n+1}}$. This does not happen for germs of corank $>1$.  
\end{abstract}
\pagenumbering{arabic}

\maketitle

\section{Introduction}

The multiple point spaces of a map-germ $(\CC^n,0)\to(\CC^p,0)$ with $n<p$ play an important r\^ole in the study of 
its geometry, as well as the topology of the image of a stable 
perturbation (\cite{mond87}, \cite{marar-mond}, \cite{goryunov-mond}). 

Formally, the $k$'th multiple point space $D^k$ of a finite proper map between topological 
spaces is the closure of the set of $k$-tuples of pairwise distinct points having 
the same image under the map.
A closed formula for an ideal defining $D^k(f)$ in some smooth ambient space
is in general not available. However, for $k=2$ the ideal 
\begin{equation}\label{d2} \mathcal{I}_2:=(f\times f)^*I_{\Delta_p}+\textnormal{Fitt}_0(
I_{\Delta_n}/(f\times f)^*I_{\Delta_p})
\end{equation}
where $I_{\Delta_n}$ and $I_{\Delta_p}$ are the ideal sheaves defining the diagonals 
$\Delta_n$ in $\mathbb{C}^n\times \mathbb{C}^n$ and $\Delta_p$ in $\mathbb{C}^p\times 
\mathbb{C}^p$, gives a scheme structure with many desirable qualities: if $f$
is dimensionally correct -- that is, if $D^2(f)$ has the expected dimension, $2n-p$,
then $D^2(f)$ is Cohen Macaulay. If moreover $f$ is finitely determined (for left-right equivalence), or, equivalently, has isolated instability, then provided its dimension is greater than $0$,
$D^2(f)$ is reduced. 

If the corank of $f$ (the dimension of $\mbox{Ker}\,df_0$), is equal to 1,
much more is possible. An explicit list of generators for the ideal defining $D^k(f)$
in $(\CC^n)^k$ was given in  
\cite{mond87} and 
\cite{marar-mond}. The second paper shows that a 
finite corank 1 map-germ $f\colon (\CC^n,0)\to(\CC^p,0)$ is
stable if and only if each $D^k(f)$ is smooth of dimension $p-k(p-n)$, or empty, 
for all $k\geq 2$. Moreover, it is 
finitely $\mathcal{A}$-determined if and only if $D^k$ is an ICIS of 
dimension $p-k(p-n)$ or empty for those $k$ with $p-k(p-n)\geq 0$, and $D^k$ 
consists at most of only the origin if $p-k(p-n)<0$ 
(see, e.g., \cite{marar}, \cite{goryunov-mond} for other results).

Any map-germ of corank 1 can be written with respect to suitable coordinates, in the
form 
\beq\label{lac}
f(\mathbf{x},y)=\bigl(\mathbf{x}, f_n(\mathbf{x},y),\ld, f_{p}(\mathbf{x},y)\bigr).
\eeq
where $\mathbf{x}:=x_1,\ld,x_{n-1}$ and $y$ together make a coordinate system on $\CC^n$. 
Provided $f$ is finite, it follows that 
the local algebra $Q(f):=\OO_{\CC^n,0}/f^*\mathfrak{m}_{\CC^{p},0}$ is isomorphic to
$\CC[y]/(y^{r+1})$, where $r+1$ is the multiplicity of $f$, and
$\OO_{\CC^{n},0}$ is minimally generated over $\OO_{\CC^{n+1},0}$ by 
$1, y,\ld, y^r.$
The main result of this paper relates a presentation of $\OO_{\CC^n}$ as $\OO_{\CC^{n+1}}$-module via $f$ to presentations of $\OO_{D^{k+1}(f)}$ as $\OO_{D^k(f)}$-module via the projection $\pi^{k+1}_k\colon D^{k+1}(f)\to D^k(f)$ which forgets the last component. Denote by $f^{(k)}$ the natural map $D^k(f)\to\CC^{n+1}$ induced by $f$. 

\begin{mainthm}\label{thmA} Suppose that 
$f\colon (\CC^n,0)\to(\CC^{n+1},0)$ is a finite and generically one-to-one map-germ of corank 1 for which $D^k(f)$ has dimension 
$n-k+1$ or is empty, and that $y\in\OO_{\CC^n,0}$ 
is a germ such that $1,y,\ld,y^r$  generate 
$\OO_{\CC^n,0}$ over $\OO_{\CC^{n+1},0}$. If \begin{equation*}  0\rightarrow 
\mathcal{O}_{\mathbb{C}^{n+1},0}^{r+1} \xrightarrow{\Lambda} \mathcal{O}_{\mathbb{C}
^{n+1},0}^{r+1} \xrightarrow{G} \mathcal{O}_{\mathbb{C}^{n},0}\rightarrow 0 
\end{equation*} is a minimal resolution of   $\mathcal{O}_{\mathbb{C}^n,0}$
with $\Lambda$ symmetric and $G=\left[\begin{array}{cccc} 1 & y & \cdots 
& y^r\end{array}\right]$, then for any $k=1,\ldots, \textnormal{min}(r,n)$, there 
is an exact sequence
\begin{equation}\label{pdk}  0\rightarrow \mathcal{O}_{D^k(f),0}^{r-k+1}
\xrightarrow{f^{(k)*}(\Lambda^k_k)} \mathcal{O}_{D^k(f),0}^{r-k+1} 
\rightarrow \mathcal{O}_{D^{k+1}(f),0} 
\rightarrow 0\end{equation} in which 
$\Lambda^k_k$ is the matrix obtained from $\Lambda$ by deleting the first $k$ rows and 
columns. 
\end{mainthm}
We will prove this theorem in Section \ref{sect-main}.

In \cite[Prop. 3.2]{klu-curv}, 
Kleiman, Lipman and Ulrich proved that for a finite map $F\colon 
X\rightarrow Y$ of locally Noetherian schemes of dimensions $n$ and $n+1$,
\begin{equation}\label{klueq1} \mathcal{O}_{D^2(F)}\cong \textnormal{Ker}\left(\mu:
\mathcal{O}_{X} \otimes_{\mathcal{O}_{Y}}\mathcal{O}_{X}  \rightarrow\mathcal{O}_{X}
\right)
\end{equation} as $\OO_X$-modules, 
(under some additional hypotheses: $F$ is of corank 1, flat dimension 1, 
and $Y$ satisfies Serre's condition ($\textnormal{S}_2$) \cite[Theorem 11.5 (i)]
{eisenbud}). Here $\mu$ is the multiplication map $a\otimes b\mapsto ab$. Moreover,
\begin{equation}\label{klueq2}\textnormal{Fitt}_i(F^*F_*\mathcal{O}_X)=\textnormal
{Fitt}_{i-1}((\pi^2_1)_*\mathcal{O}_{D^2(F)})\end{equation} for all $i\geq 1$ (\cite[Lemma 3.9]
{klu-curv}). This equality, for a finite and generically one-to-one map-germ of 
corank 1 from $\CC^n$ to $\CC^{n+1}$, is an easy consequence of our main theorem.

In Proposition \ref{pkerm} we prove that for a finite and generically one-to-one map-germ 
$f\colon (\CC^n,0)\to(\CC^{n+1},0)$
of any corank, $\ker(\mu)$ 
has resolution
\begin{equation}\label{sequence0} \begin{CD}
0 & @>>> & \mathcal{O}_{\mathbb{C}^{n},0}^{r} & @>f^*\Lambda^1_1>> &  \mathcal{O}_{
\mathbb{C}^{n},0}^{r} & @>>> & \textnormal{Ker}(\mu)& @>>>  &0
\end{CD} \end{equation}
This, together with the main theorem, 
provides an alternative proof of 
(\ref{klueq1}) for generically one-to-one map-germs of corank 1 in 
dimensions $(n,n+1)$.  

In the corank $\geq 2$ case,  $\textnormal{Ker}(\mu)$ and  
$\mathcal{O}_{D^2(f)}$ are no longer isomorphic. This is evident from the fact 
that $\textnormal{Ker}(\mu)$ is a Gorenstein $\mathcal{O}_{\mathbb{C}^n,0}$-algebra, 
for it has a presentation given by a symmetric matrix (\cite[Theorem 2.3]
{kleiman-ulrich}), whereas $\mathcal{O}_{D^2(f)}$ is not Gorenstein.
However, there exists a map between the 
resolutions of $\mathcal{O}_{D^2(f)}$ and $\mathcal{O}_{\mathbb{C}^{n},0} 
\otimes_{\mathcal{O}_{\mathbb{C}^{n+1}},0}\mathcal{O}_{\mathbb{C}^n,0}$ over 
$\mathcal{O}_{\mathbb{C}^n,0}$ (see Appendix \ref{app}).
Examples suggests that a comparison of the form (\ref{klueq2}) should hold for $i=1,2$ in the case of map-germs of corank $\geq 2$ that are generically one-to-one (see Example \ref{ex-fitpf1}).
In \cite{altintas}, we proved that $\textnormal{Fitt}_1(f^*f_*\mathcal{O}_{\mathbb{C}^n,0})=\textnormal{Fitt}_{0}((\pi^2_1)_*\mathcal{O}_{D^2(f)})$ for finitely $\mathcal{A}$-determined corank $\geq 2$ map-germs in $\mathcal{E}_{3,4}^0$. However, it is still an open problem for the general case. 

This work is based on a part of the first author's PhD thesis submitted at the University of Warwick in 2011. She thanks the university for the financial support for during her stay and the department of Mathematics for the warm welcome.

\section{Background} 

\subsection{\texorpdfstring{$\mathcal{A}$-}{A}equivalence.} 

Let $\mathcal{E}_{n,p}^0$ denote the space of holomorphic map-germs from $(\mathbb{C}^n,0)$ to $(\mathbb{C}^p,0)$. The group $\mathcal{A}:=\textnormal{Diff}(\mathbb{C}^n,0)\times \textnormal{Diff}(\mathbb{C}^p,0)$ of local diffeomorphisms acts on $\mathcal{E}_{n,p}^0$ by $(\phi,\psi)\cdot f=\psi \circ f \circ \phi^{-1}$. A map-germ is $\mathcal{A}$-stable if any of its unfoldings is parametrised-equivalent to a trivial unfolding of the form $f\times 1$.
A map-germ $f\in \mathcal{E}_{n,p}^0$ is $\ell$-$\mathcal{A}$-\textit{determined} if 
every $g\in\mathcal{E}_{n,p}^0$ with the same $\ell$-jet as $f$ is $\mathcal{A}
$-equivalent to $f$. Furthermore, $f$ is  \textit{finitely $\mathcal{A}$-determined}, 
or $\mathcal{A}$-\textit{finite}, if it is $\ell$-$\mathcal{A}$-determined for some 
$\ell<\infty$. By fundamental results of Mather, finite determinacy is equivalent to the finite dimensionality of $T^1_{\mathcal A_e}f:=f^*(\theta_{\CC^p})/tf(\theta_{\CC^n})+f^{-1}(\theta_{\CC^p})$, and thus (if $f$ is not stable) to $0\in\CC^p$ being an isolated point of instability of $f$. 

\subsection{Multiple point spaces. }\label{section-multi}
Given a map $f:X\to Y$, we set 
\beq\label{off} \DD^k(f)=\{(x_1,\ld,x_k)\in X^k| f(x_1)=\cd=f(x_k), x_i\neq x_j \text{ if } i\neq j\}\eeq
and define the $k$'{\it th multiple point space of} $f$, $D^k(f)$, by
\beq\label{cl}D^k(f)=\text{closure }\DD^k(f)\eeq
(where the closure in taken in $X^k$)
{\it provided $\DD^k(f)$ is not empty}. We extend this definition to germs of maps by taking the limit over representatives; if $f\in \mathcal{E}_{n,p}^0$ is finite, the local conical structure guarantees that we obtain in this way a well defined germ at $\mathbf{0}\in (\CC^n)^k$. 
We will shortly endow $D^k(f)$ with an analytic structure in case $f$ is a germ of corank 1. This structure will be compatible with unfolding, and this will lead to a definition of $D^k(f)$ even in cases where $\DD^k(f)$ is empty.  For convenience of notation we sometimes write $D^1(f):=(\mathbb{C}^n,0)$ and $D^0(f):=(\mathbb{C}^p,0)$.

 For any germ $f$ of the form \eqref{lac}, ${D}^k(f)$ can obviously be embedded (as a set) in $(\mathbb{C}^{n-1}\times \mathbb{C}^k,0)$; we give it an analytic structure by means of the ideal
 $\mathcal{I}_k(f)$ generated over $\mathcal{O}_{\mathbb{C}^{n-1}\times\mathbb{C}^k,0}$  by $(k-1)(p-n+1)$ functions $R_i^j$, for $i=1,\ldots, k-1$ and $j=n,\ldots,p$, which are defined iteratively by
\begin{equation}\label{idk} \begin{split} R_1^j(\textbf{x},y_1,y_2)&= \frac{ f_j(\textbf{x},y_2)-f_j(\textbf{x},y_1) }{y_2-y_1}\quad \quad \textnormal{and} \\ R_{i}^j(\textbf{x},y_1,\ldots, y_{i+1})&=\frac{  R_{i-1}^j(\textbf{x},y_1,\ldots, y_{i-1},y_{i+1})-R_{i-1}^j(\textbf{x},y_1,\ldots, y_{i-1},y_{i} )}{y_{i+1}-y_{i}}. \end{split}
\end{equation}
There is a natural action of $S_k$ on ${D}^k(f)$ which permutes the coordinates $y_1,\ldots, y_k$. See \cite[Proposition 3.6]{mond87} for a set of $S_k$-invariant generators for $\mathcal{I}_k(f)$.

Only in the case where it is $0$-dimensional does the set defined by this ideal differ from the 
set defined by \eqref{cl}; the latter is evidently empty in this case. 

For a stable map-germ, $\mathcal{I}_k(f)$ is reduced (\cite[Proposition 2.14]{marar-mond}) for all $k$. 
In \cite[Proposition 2.16]{marar-mond} it is shown that the following definition, which also deals with the case where $D^k(f)$ is zero-dimensional, is compatible with the definition in terms of the ideal ${\mathcal I}_k(f)$ described above in case $\dim\ D^k(f)>0$. 
\begin{defn}[Proposition 2.5 \cite{houston09}, cf. \cite{gaffney-multiple}]\label{multiple} Let $f\in \mathcal{E}_{n,p}^0$ be a finite map-germ of corank 1. Let $F\colon (\mathbb{C}^n\times \mathbb{C}^d,0)\rightarrow(\mathbb{C}^p\times \mathbb{C}^d,0)$ be a stable unfolding of $f$, with $F(\textbf{x},y,\textbf{u}):=(\tilde{F}_{\textbf{u}}(\textbf{x},y),\textbf{u})$ and $\tilde{F}_0(\textbf{x},y)=f(\textbf{x},y)$. Then we set 
\begin{equation} D^k(f)={D}^k(F)\cap \{{u}_1=0,\ldots,{u}_k=0\} \end{equation} 
where ${u}_1,\ldots, {u}_k$ are the unfolding parameters, and $D^k(F)$ is given its reduced structure.
\end{defn} 
It is straightforward to check that this is independent of the choice of stable unfolding, and is compatible with unfolding in the sense that for {\it any} germ of unfolding $F:(\CC^n\times \CC^d,0)\to (\CC^p\times\CC^d,0)$ of 
$f$, the diagram 
\beq\label{fs}
\xymatrix{D^k(f),{0}\ar[d]\ar[r]&D^k(F),{0}
\ar[d]\\
\{0\}\ar[r]&\CC^d,0}\eeq
in which the vertical arrows are projections to the base and the horizontal arrows are inclusions, is a fibre square. 

For a finite $f\in\mathcal{E}_{n,p}^0$, the \textit{$k$'th multiple point space on 
the target} is the set
$$M_k(f)=\{y\in(\mathbb{C}^p,0) \mid |f^{-1}(y)|\geq k\}$$
(where preimages are counted with multiplicity)
with analytic structure defined by the $(k-1)$'st Fitting ideal 
$\textnormal{Fitt}_{k-1}(f_*\mathcal{O}_{\mathbb{C}^n,0})$.   

For $k>\ell$ we define $D^k_\ell(f)$ to be the image in $D^k(f)$ of $D^\ell(f)$ under the composite $\pi^{\ell+1}_\ell\circ\cd\circ\pi^k_{k-1}$.
Then we have set-theoretic equalities 
$f^{(k)}(D^k(f))=M_k(f)$ and $f^{-1}M_k(f)=D^k_1(f)$ for all $k\geq 1$.

We denote by $q(f)$ the multiplicity of $f$, $q(f):=\dim_{\CC}Q(f)$.

If $f\in \mathcal{E}_{n,n+1}^0$  is finite and $Q(f)$ has $\CC$-basis 
$g_0=1,g_1,\ldots,g_r$, 
then there is an induced resolution of $\mathcal{O}_{\mathbb{C}^n,0}$ over 
$\mathcal{O}_{\mathbb{C}^{n+1},0}$ of length 1 of the form
\begin{equation*}  0\rightarrow \mathcal{O}_{\mathbb{C}^{n+1},0}^{r+1} \xrightarrow{\Lambda} \mathcal{O}_{\mathbb{C}^{n+1},0}^{r+1} \xrightarrow{G} \mathcal{O}_{\mathbb{C}^{n},0}\rightarrow 0 \end{equation*}
where $G:=\begin{bmatrix} 1 & g_1 & \ldots & g_r\end{bmatrix}$
and  $\Lambda$ is an $(r+1)\times(r+1)$-matrix with entries $\lambda_{ij} \in \mathcal{O}_{\mathbb{C}^{n+1},0}$ for $i,j=0,\ldots, r$.  Moreover, $\Lambda$ can be chosen to be symmetric (\cite{mond-pellikaan}). In this setup, $\textnormal{Fitt}_{k-1}(f_*\mathcal{O}_{\mathbb{C}^n,0})$ is  the ideal of $(r-k+1)\times(r-k+1)$-minors of $\Lambda$.

\begin{prop} \label{propqf} Let $f\in\mathcal{E}_{n,p}^0$ be a finite analytic map-germ of corank 1. Then 
$D^k(f)=\emptyset$ for $k>q(f)$, and $q(\pi^{k}_{k-1})=q(f)-k+1$ for $k=1,\ld, q(f)$.\end{prop}

\begin{proof}
Suppose $f$ is written in the form \eqref{lac}. Then
\begin{eqnarray*} Q(\pi^{k}_{k-1}) &\cong & \frac{\mathcal{O}_{\mathbb{C}^{n-1}\times \mathbb{C}^{k},0}}{\mathcal{I}_{k}(f)+(\textbf{x},y_1,\ldots,y_{k-1})}. \end{eqnarray*}  Since $R_i^j$ belongs to $(\textbf{x},y_1,\ldots,y_{i+1})\mathcal{O}_{\mathbb{C}^{n-1}\times \mathbb{C}^{k+1},0}$ for $1\leq i\leq k-1$ and $n\leq j\leq p$, 
\begin{align*}\mathcal{I}_{k}(f)+
(\textbf{x},y_1,\ldots,y_{k-1}) =& ( R^n_{k-1}(0,\ldots,0, y_{k}), \ldots,  R^p_{k-1}(0,\ldots,0,y_{k})) +\\ &+(\textbf{x},y_1,\ldots,y_{k-1}).  \end{align*}
By an induction, we find
$R^j_{s}(0,\ldots,0,y_{k})=f_j(0,y_{k})/y_{k}^s$. Now the result follows from a straightforward calculation. 
\end{proof}

For a finite map-germ of any corank, $\pi_{k-1}^{k}$ is clearly finite if $D^k(f)$ and $D^{k-1}(f)$ are given their reduced structure. However the relation between these multiplicities is less straightforward than in \ref{propqf} if the corank is not 1. 
\begin{prop}\label{qfqpi} Let $f\in \mathcal{E}_{n,p}^0$ be a finite map-germ of corank $\geq 2$. Then $$q(\pi^2_1)\leq q(f).$$
\end{prop}
\begin{proof} Consider \eqref{d2}. Denoting the Fitting ideal in the denominator simply by $\text{Fitt}_0$, we have
\begin{eqnarray*} Q(\pi^2_1) &\cong & \frac{\mathcal{O}_{\mathbb{C}^{n}\times \mathbb{C}^{n},0}}{(f_1(\textbf{x}_1)-f_1(\textbf{x}_2),\ldots, f_p(\textbf{x}_1)-f_p(\textbf{x}_2))+\textnormal{Fitt}_0+(\textbf{x}_1)}
\cr
&\cong&  \frac{\mathcal{O}_{\mathbb{C}^{n},0}}{(f_1(\textbf{x}_2),\ldots,f_p(\textbf{x}_2))+\textnormal{Fitt}_0|_{\textbf{x}_1=0}}
\end{eqnarray*} This is a quotient of $Q(f)$.  
\end{proof}

\begin{example}\label{ex1} For the stable map-germ 
$$h\colon (\textbf{u},y,z)\mapsto (\textbf{u},y^2+u_1z,z^2+u_2y,u_3yz+u_4y+u_5z),$$
we have $Q(h)=\mathbb{C}\cdot \{1,y,z,yz\}$ and $Q(\pi^2_1(h))=\mathbb{C}\cdot \{1,y_2,z_2\}$.
\end{example}

\begin{example}\label{ex2} For the stable map-germ
$$\tilde{h}\colon (\textbf{u},y,z)
\mapsto(\textbf{u}, y^{3}+u_1z+{u}_{2} y, yz+{u}_{3} z,
z^2+{u}_{4} y+{u}_{5} y^{2}+{u}_{6} z)$$ 
we find $Q(\tilde{h})=\mathbb{C}\cdot \{1,y,y^2,z\}$ and $Q(\pi^2_1(\tilde{h}))=\mathbb{C}\cdot \{1,y_2,y_2^2,z_2\}$.
\end{example}

\subsection{Pull-backs of map-germs}

By Mather's results (\cite{matherVI}), if $f\in\mathcal{E}_{n,p}^0$ is finite then there exists a stable germ $F\colon (\mathbb{C}^{N},0) \rightarrow (\mathbb{C}^{P},0)$ and a germ of immersion $g\colon (\mathbb{C}^p,0)\rightarrow (\mathbb{C}^{P},0)$ with $g$ transverse to $F$, such that $f$ is obtained as a fibre product
\begin{equation*} \xymatrix{ (\mathbb{C}^{N},0) \ar[r]^F & (\mathbb{C}^{P},0) \\ (\mathbb{C}^{n},0)\cong (\mathbb{C}^N\times_{\mathbb{C}^P}\mathbb{C}^p,0)
\ar[u]^i \ar[r]^>>>>>f & (\mathbb{C}^{p},0) \ar[u]^g}\end{equation*} and $P-N=p-n$. Naturally, $\mathcal{O}_{\mathbb{C}^N\times_{\mathbb{C}^P} \mathbb{C}^p,(0,0)}\cong \mathcal{O}_{\mathbb{C}^N,0}\hat{\otimes}_{\mathcal{O}_{\mathbb{C}^P,0}} \mathcal{O}_{\mathbb{C}^p,0}$. In fact, we have
 $\mathcal{O}_{\mathbb{C}^N\times_{\mathbb{C}^P} \mathbb{C}^p,(0,0)}\cong \mathcal{O}_{\mathbb{C}^N,0}\otimes_{\mathcal{O}_{\mathbb{C}^P,0}} \mathcal{O}_{\mathbb{C}^p,0}$ since $g$ is finite (\cite[Lemma 1.89]{greuel-l-s}).

We note that
\begin{equation*}
\textnormal{Fitt}_j(f_*\mathcal{O}_{\mathbb{C}^n,0}) = g^*\textnormal{Fitt}_j(F_*\mathcal{O}_{\mathbb{C}^N,0})\end{equation*}
since Fitting ideals commute with base change (\cite[\S 1]{teissier}). Indeed we have
\begin{prop}\label{glambda} Suppose $X$ is a germ of dimension $n$. Let  $f\colon X\to (\CC^{n+1},0)$ be a pull-back of a finite $F\in\mathcal{E}_{N,N+1}^0$ by $g\in \mathcal{E}_{n+1,N+1}^0$. 
If \begin{equation*}0\rightarrow \mathcal{O}_{\mathbb{C}^{N+1},0}^{r+1} \xrightarrow{\Lambda} \mathcal{O}_{\mathbb{C}^{N+1},0}^{r+1} \rightarrow \mathcal{O}_{\mathbb{C}^{N},0} \rightarrow 0 \end{equation*} is minimal so is
$$0\rightarrow \mathcal{O}_{\mathbb{C}^{n+1},0}^{r+1} \xrightarrow{g^*\Lambda} \mathcal{O}_{\mathbb{C}^{n+1},0}^{r+1} \rightarrow \mathcal{O}_{X} \rightarrow 0.$$
\end{prop}
\begin{proof} This is an easy consequence of the right exactness of tensor products and the fact that
 $\det (g^*\Lambda)$ is identically $0$ only if $\textnormal{im}(g)\subseteq \textnormal{im}(F)$, which contradicts the assumption on the dimension of $X$.
 \end{proof}

%%%%%%%%%%%%%%%%%%%%%%%%%%%%%%%%%%%%%%%%%%%%%%
\subsection{Principle of iteration}

The setup of multiple point spaces allows us to define certain \textit{iterations} involving the projections 
$\pi^k_{k-1}$ and the map-germ itself. First, there is a natural bijection 
\beqa \phi\colon D^s(\pi^k_{k-1}(f)) & \rightarrow & D^{k+s-1}(f) \cr
\left((\textbf{x}_1^1,\ldots,\textbf{x}_k^1),\ldots,(\textbf{x}_1^s,\ldots,\textbf{x}_k^s)\right) &\mapsto  & \left((\textbf{x}_1^1,\ldots,\textbf{x}_{k-1}^1,\textbf{x}_k^1),\ldots,(\textbf{x}_1^1,\ldots,\textbf{x}_{k-1}^1,\textbf{x}_k^s)\right) \nonumber
\eeqa
with $\textbf{x}^i_j\in \mathbb{C}^n$ for all $i,j$ (\cite[Remark 2.7 (iii)]{goryunov-mond}).
Recall that if $f$ is stable of corank 1 then $D^k(f)$ is smooth (or empty) for all $k$. Moreover,  $\pi^k_{k-1}$ is also stable and of corank 1 for all $k$ (\cite{goryunov}). So, $\phi$ induces
\beq \label{iprince} \mathcal{O}_{D^s(\pi^k_{k-1}(f))}\cong \mathcal{O}_{D^{k+s-1}(f)}.
\eeq

The isomorphism (\ref{iprince}) can be extended to finite map-germs of corank 1. First we note,

\begin{lem}\label{punfold} Let  $f\in \mathcal{E}_{n,p}^0$ be a finite map-germ of corank 1. Let $F\in \mathcal{E}_{n+d,p+d}^0$  be a parametrised unfolding of $f$. Then $\pi^k_{k-1}(F)$ is an unfolding of $\pi^k_{k-1}(f)$.
\end{lem}
\begin{proof} By definition, \begin{equation*}
\mathcal{O}_{D^k(f)}\cong \mathcal{O}_{D^k(F)}/\mathfrak{m}_{\mathbb{C}^d,0}.\end{equation*} 
Moreover, an unfolding is a level-preserving map-germ whose restriction to the zero fibre is the original map-germ (\cite[Section 1]{mond-montaldi}). The result is now straightforward.
\end{proof}

\begin{prop} \label{finprince} Let  $f\in \mathcal{E}_{n,p}^0$ be a finite map-germ of corank 1. Then $\mathcal{O}_{D^s(\pi^k_{k-1})}\cong \mathcal{O}_{D^{k+s-1}(f)}$.
\end{prop}
\begin{proof}  Let $F\in \mathcal{E}_{n+d,p+d}^0$  be a parametrised unfolding of $f$. Combine the result of Lemma \ref{punfold} with (\ref{iprince}) to get
\begin{equation*} \mathcal{O}_{D^s(\pi^k_{k-1}(f))}\cong \mathcal{O}_{D^s(\pi^k_{k-1}(F))}/\mathfrak{m}_{\mathbb{C}^d,0}\cong \mathcal{O}_{D^{k+s-1}(f)}.\end{equation*}
\end{proof}

%%%%%%%%%%%%%%%    NEW PROOFS    %%%
\section{Presenting \texorpdfstring{$\OO_{D^{k+1}(f)}$}{O{D(k+1)}} over \texorpdfstring{$\OO_{D^k(f)}$}{OD(k)}.}
Throughout the first part of this section we suppose that 
$f:(\CC^n,0)\to(\CC^{n+1},0)$ is stable of corank 1 and multiplicity $r+1$. 
Writing $f$ with respect to  
linearly adapted coordinates, for each $k$ we 
embed $D^k(f)$ into $\CC^{n-1}\times \CC^k$, as explained in 
Subsection \ref{section-multi}.  On $\CC^{n-1}\times\CC^k$ we take coordinates 
$x_1,\ld,x_{n-1}, y_1,\ld, y_k$, or simply $\mathbf{x},y_1,\ld,y_k$.
Finally, we assume given 
a symmetric matrix $\Lambda=(\lambda_{ij})$ 
presenting $\OO_{\CC^n}$ over 
$\OO_{\CC^{n+1}}$ with respect to generators 
$1,y,\ld,y^r$.  

Now let $\Lambda^k_k$ be the matrix $\Lambda$ with its first $k$ rows and columns deleted, and let $\Lambda^k$ be the matrix $\Lambda$ with its first $k$ rows, but not columns,  deleted.
We will see that for each $k$, $\Lambda^k_k$ is the matrix of a presentation of 
$\OO_{D^{k+1}(f)}$ over $\OO_{D^k(f)}$, with respect to generators obtained from 
$1,y,\ld,y^r$ by a procedure involving iterated interpolation and division. 
Before giving the proof for all $k$, we show it in the case $k=1$. 

From each relation 
\beq\label{1}\lambda^0_i(f(\mathbf{x},y))+\lambda^1_i(f(\mathbf{x},y))y+\cd+\lambda^r_i(f(\mathbf{x},y))y^r=0\eeq
we obtain
\beq\label{6} \lambda^0_i(f(\mathbf{x},y_2))+\lambda_i^1(f(\mathbf{x},y_2))y_2+\cd+\lambda_i^r(f(\mathbf{x},y_2))y_2^r=0\eeq
and also, if $(\mathbf{x},y_1,y_2)\in D^2(f)$, so that $f(\mathbf{x},y_1)=f(\mathbf{x},y_2)$, 
\beq\label{7} \lambda^0_i(f(\mathbf{x},y_1))+\lambda_i^1(f(\mathbf{x},y_1))y_2+\cd+\lambda_i^r(f(\mathbf{x},y_1))y_2^r=0.\eeq
Subtracting \eqref{7} from \eqref{6}, we get
$$ \lambda_i^1(f(\mathbf{x},y_1))(y_1-y_2)+\lambda_i^2(f(\mathbf{x},y_1))(y_1^2-y_2^2)+\cd+\lambda_i^r(\mathbf{x},y_1))(y_1^r-y_2^r)=0.$$
Dividing by $y_1-y_2$, which is not a zero-divisor on $D^2(f)$
since the intersection of $D^2(f)$ with $\{y_1=y_2\}$ is the non-immersive locus, which has dimension $n-2$, less than the dimension of $D^2(f)$, we get
\beq\label{8}
\lambda_i^1(f(\mathbf{x},y_1))+\lambda_i^2(f(\mathbf{x},y_1))(y_1+y_2)+\cd+\lambda_i^r(f(\mathbf{x},y_1))(y_1^{r-1}+\cd+y_2^{r-1})=0.\eeq
Thus the columns of $f^*\Lambda^1$, are relations among the set
\beq\label{9}1,y_1+y_2,  y_1^2+y_1y_2+y_2^2, \ld, y_1^{r-1}+\cd+y_2^{r-1}\eeq
of elements of $\OO_{D^2(f)}$. These generate $\OO_{D^2(f)}$ over $\OO_{\CC^n}$, minimally, by \ref{propqf}. In fact the relation given by  the first column of $f^*\Lambda^1$ is a linear combination over $\OO_{\CC^n}$ of the rest: by the symmetry of $f^*\Lambda$, from 
$(1,y, \ld,y^r)(\Lambda\circ f) =0$, 
$(\Lambda\circ f)(1,y,\ld,y^r)^t=0$
and therefore
\beq\label{k2}(\Lambda\circ f\circ \pi^2_1)(1,y,\ld,y^r)^t\circ\pi^2_1=0.\eeq 
This last argument is not needed in the proof which now follows, but perhaps helps to explain the result. 
\begin{lem}\label{fp} $f^*\Lambda^1_1$ is a presentation of $\OO_{D^2(f)}$ over $\OO_{\CC^n}$. 
\end{lem}
\begin{proof} Consider the diagram
\beq\label{nub}
\xymatrix{0\ar[r]&\OO_{\CC^n}^r\ar[r]^{A}&\OO_{\CC^n}^r\ar[r]&\OO_{D^2(F)}\ar[r]&0\\
0\ar[r]&\OO_{\CC^n}^r\ar[r]^{f^*\Lambda^1_1}\ar[u]_{\alpha_1}&\OO_{\CC^n}^r\ar[r]\ar[u]_{\alpha_0}&\text{coker }f^*\Lambda^1_1\ar[r]\ar[u]_{\alpha}&0}
\eeq
in which 
\begin{itemize}
\item
$A$ is a presentation of $\OO_{D^2(f)}$ with respect to the generators \eqref{9}, which we may assume square, and which is certainly minimal;
\item
$\alpha$, sending the class of the $i$'th generator of $\OO_{\CC^n}^r$  to the $i$'th  member of the set \eqref{9},  is an epimorphism by virtue of the fact that the members of \eqref{9} generate $\OO_{D^2(f)}$,
\item
 $\alpha_0$ and $\alpha_1$ are successive lifts of $\alpha$,
 \end{itemize}
 Because $\alpha$ is an epimorphism, so is $\alpha_0$, and its injectivity follows. Up to multiplication by a unit, $\det A=\det f^*\Lambda^1_1$. For $\det f^*\Lambda_1^1$ generates the conductor of $\OO_{\CC^n}$
 into $\OO_{f(\CC^n)}$, by \cite[Theorem 3.4 and Lemma 3.3]{mond-pellikaan}; on the other hand, $\det\,A$ generates the conductor because it defines the image of $\pi_1^2$, with reduced structure since 
 $\pi^2_1$ is generically one-to-one. Now since $\det\alpha_0$ is a unit, it follows by commutativity of the diagram 
 that $\det \alpha_1$ is also
 a unit. Thus $\alpha$ is an isomorphism and $f^*\Lambda_1^1$ is a presentation of $\OO_{D^2(f)}$. 
 \end{proof}
 We now go on to prove inductively that for each $k$ with $2\leq k\leq r$, $f^*\Lambda^k_k$ is the matrix of a presentation of $\OO_{D^{k+1}(f)}$ over $\OO_{D^k(f)}$.  At the induction step we produce a set of
 generators for $\OO_{D^{k+1}(f)}$ over $\OO_{D^k(f)}$, from the generators for $\OO_{D^k(f)}$ over $\OO_{D^{k-1}(f)}$, by the same procedure -- interpolation and division -- with which we obtained the set 
 \eqref{9}.
 \vskip 10pt
 {\bf New generators:} We apply this procedure to the set \eqref{9}: for each member $g^{(2)}_i$ except the first (where it would give 0), we form the quotient 
 \beq\label{quot}
g^{(3)}_{i-1}:=\frac{g^{(2)}_i(x,y_1,y_2)-g^{(2)}_i(x,y_1,y_3)}{y_2-y_3},\eeq
by this means obtaining the set 
 \beq\label{10} 1, y_1+y_2 +y_3, \ld, \frac{\frac{y_1^r-y_2^r}{y_1-y_2}-\frac{y_1^r-y_3^r}{y_1-y_3}}{y_2-y_3}. \eeq
 Note that the cardinality has decreased by 1.
 \begin{lem}\label{newgens} The set $g_0^{(k)},\ld,g_{r-k+1}^{(k)}$ obtained by iterating this procedure $k-1$ times, 
 is a minimal generating set for $\OO_{D^k(f)}$ over $\OO_{D^{k-1}(f)}$. Moreover, for each $i$, $g^{(k-1)}_i$ is a sum of all monomials of degree $i$ in $y_1,\ld,
 y_{k-1}$, and in particular $g^{(k-1)}_0=1$.
 \end{lem}
 \begin{proof}  The second statement is easily proved by induction. The first statement follows: 
in $\OO_{D^k(f)}$ we have 
 \beqa\text{Sp}_{\CC}\{g_0^{(k)},\ld,g_{r-k+1}^{(k)}\}+(\pi^k_{k-1})^*\mathfrak{m}_{D^{k-1}(f),0}
 &=&\text{Sp}_{\CC}\{g_0^{(k)},\ld,g_{r-k+1}^{(k)}\}+\left(\mathbf{x},y_1,\ld,y_{k-1}\right)
 \cr
 &=&\text{Sp}_{\CC}\{1,y_k,\ld,y_k^{r-k+1}\}+\left(\mathbf{x},y_1,\ld,y_{k-1}\right)
 \cr
 &=&\OO_{D^k(f)}
 \eeqa
 \end{proof}
\begin{lem}\label{newrelns} Each relation 
  \beq\label{relo} \mu_0 g_0^{(k)}+\mu_1g_1^{(k)}+\cd +\mu_{r-k+1}g_{r-k+1}^{(k)}=0
  \eeq
  with the $\mu_i\in\OO_{D^{k-1}(f)}$ gives rise to a relation 
  \beq\label{reln}
  \mu_1 g_0^{(k+1)}+\cd+\mu_{r-k+1}g^{(k+1)}_{k-r}=0
 \eeq
 among the generators of $\OO_{D^{k+1}(f)}$ over $\OO_{D^k(f)}$.
 \end{lem}
 \begin{proof}
 The procedure is exactly what we described in equations \eqref{1} through \eqref{8}.
If $(x,y_1,\ld,y_{k+1})\in D^{k+1}(f)$ then subtracting \eqref{relo}, evaluated at $(\mathbf{x},y_1,\ld,y_{k-1},y_{k+1})$, from its evaluation at $(\mathbf{x},y_1,\ld,y_k)$, we kill the first term. Note that the values of the $\mu_i$ do not change, since $\pi^{k}_{k-1}(\mathbf{x},y_1,\ld,y_{k-1},y_{k+1})=\pi^{k}_{k-1}(\mathbf{x},y_1,\ld,y_{k})$. Now dividing by $y_{k}-y_{k+1}$, we obtain 
\eqref{reln}.
\end{proof} 
%%%
\begin{thm}\label{thmstable} 
$f^*\Lambda^k_k$ is a presentation of $\OO_{D^{k+1}(f)}$ over $\OO_{D^k(f)}$ with respect
to the generators $g^{(k+1)}_0,\ld,g^{(k+1)}_{r-k}$ for all $k\leq \textnormal{min}(r,n)$.
\end{thm}
\begin{proof} Induction on $k$. Once again, the induction step follows the proof of  the case $k=2$. Assume  $k\leq \textnormal{min}(r,n)$ so that $\textnormal{dim }D^{k+1}(f)\geq 0$ and $\Lambda_{k-1}^{k-1}$ has size $\geq 1$. Suppose
inductively that $f^*\Lambda^{k-1}_{k-1}$ is a presentation of $\OO_{D^{k}(f)}$ over $\OO_{D^{k-1}(f)}$.  Since $\OO_{D^{k+1}(f)}$ is a graded module -- $D^{k+1}(f)$ is smooth -- we may assume that this presentation is given with respect
to the generators $g^{(k)}_0,\ld,g^{(k)}_{r-k+1}$. Then by Lemmas \ref{newgens} and \ref{newrelns},  
$f^{(k)*}\Lambda^{k}_k$ presents an $\OO_{D^k(f)}$-module which maps epimorphically to $\OO_{D^{k+1}(f)}$. Let $A$ be a square, minimal presentation of $\OO_{D^{k+1}(f)}$ with respect to $g^{(k+1)}_0,\ld,g^{(k+1)}_{r-k}$. The argument of the proof of \ref{fp}, applied to the analogous diagram 
\beq\label{nib}
\xymatrix{0\ar[r]&\OO_{D^k(f)}^{r-k}\ar[r]^{A}&\OO_{D^k(F)}^{r-k}\ar[r]&\OO_{D^{k+1}(f)}\ar[r]&0\ \ \\
0\ar[r]&\OO_{D^k(f)}^{r-k}\ar[r]^{\Lambda^k_k}\ar[u]_{\alpha_1}&\OO_{D^k(f)}^{r-k}\ar[r]\ar[u]_{\alpha_0}&\text{coker }\Lambda^k_k\ar[r]\ar[u]_{\alpha}&0\ ,}
\eeq
proves that $f^{(k)*}\Lambda^{k}_k$ is a presentation of $\OO_{D^{k+1}(f)}$ over $\OO_{D^{k}(f)}$. In fact, $f^{(k)*}\Lambda^{k}_k=f^{*}\Lambda^{k}_k$ since $\pi^{k}_{k-1}(f)$ is only a projection onto the first $(k-1)$ components.
\end{proof} 

\subsection{Proof of the main theorem}\label{sect-main} 
Assume that $q(f)=r+1$. Let $F\in \mathcal{E}_{n+d,n+d+1}^0$ be a parametrised stable unfolding of $f$.  Then, $\pi^{k+1}_k(F)$ is an unfolding of $\pi^{k+1}_k(f)$ and
\beq D^{k+1}(f)\cong D^{k+1}(F)\times_{D^{k}(F)} D^{k}(f).\eeq
By an induction on $k$ and the principal of iteration, we deduce that $\pi^{k+1}_k(f)$ is also generically one-to-one and that $D^{k+1}(f)$ has the expected dimension, $n-k$. 
By Theorem \ref{thmstable}, $F^*\Lambda^k_k$ presents  $\mathcal{O}_{D^{k+1}(F)}$ over $\mathcal{O}_{D^k(F)}$ for all $k\leq \textnormal{min}(r,n)$.
Hence, Proposition \ref{glambda}, applied to $\pi^{k+1}_k(f)$, shows that 
\beq\label{kst} 
\xymatrix{0\ar[r] & \OO_{D^{k}(f)}^{r-k+1}\ar[rr]^{(F^*\Lambda^{k}_{k})|_{\textbf{u}=0}} && \OO_{D^{k}(f)}^{r-k+1}\ar[r] & \OO_{D^{k+1}(f)} \ar[r] &0.}
\eeq
is exact. Since $\Lambda|_{\textbf{u}=0}$ is a presentation of $\mathcal{O}_{\mathbb{C}^n,0}$ over $\mathcal{O}_{\mathbb{C}^{n+1},0}$, the main result follows.\eop

\vskip8pt
\noindent Consequently, 
\begin{cor}[cf.\ Lemma 3.9, \cite{klu-curv}]\label{fitpf} Let $f\in \mathcal{E}_{n,n+1}^0$ be finite and generically one-to-one map-germ of corank 1. Then \begin{equation}\label{efitpf} \textnormal{Fitt}_{i}(f^*f_*\mathcal{O}_{\mathbb{C}^n,0})=\textnormal{Fitt}_{i-1}((\pi^2_1)_*\mathcal{O}_{D^2(f)})
\end{equation} for $i\geq 1$.
\end{cor}

The following examples show that  for map-germs of corank $\geq 2$,  in general (\ref{efitpf}) holds only for $i=1,2$. 
%%%%
\begin{example}\label{ex-fitpf1} For the stable map-germ of Example \ref{ex1},
one calculates that
 $$\textnormal{Fitt}_{i}(h^*h_*\mathcal{O}_{\mathbb{C}^7,0})=
\textnormal{Fitt}_{i-1}((\pi^2_1(h))_*\mathcal{O}_{D^2(h)})$$ for $i=1,2$, 
but
$$\textnormal{Fitt}_{j}({h}^*{h}_*\mathcal{O}_{\mathbb{C}^n,0})\neq 
\textnormal{Fitt}_{j-1}((\pi^2_1({h}))_*\mathcal{O}_{D^2({h})})$$
 for $j=3,4$.
The same equalities, and inequalities, hold for the map-germ $\tilde h$ of Example 
\ref{ex2}.
\end{example}

Another corollary to the main theorem is the following.
\begin{cor}\label{freediv} With the hypotheses of the main theorem, if $f$ is stable, then
\beq S:=\{\det (f^*\Lambda^{k-1}_{k-1})\cdot \det (f^*\Lambda^k_k)=0\}\eeq is a 
free divisor in $D^k(f)$ for $k=1,\ldots, \textnormal{min}(r,n)$ 
(where we take $\Lambda^0_0=\Lambda $).\end{cor}

\begin{proof} This is trivial for $k=1$:  $S=(\mathbb{C}^n,0)$. Assume 
that  $2\leq k \leq \textnormal{min}(r,n)$. Then, $\mathcal{O}_{D^{k+1}(f)}$ is 
presented over $\mathcal{O}_{D^k(f)}$ by $f^*\Lambda^k_k$. Since $f$ is 
stable, $\pi^{k+1}_k$ is also stable of corank 1. Therefore, the statement 
follows by \cite[Theorem 1.2]{mond-schulze} which states that if  
$F:(\CC^r,0)\to(\CC^{r+1},0)$ is stable and of corank 1, and if $\Lambda$ is a symmetric
presentation of $F_*\OO_{\CC^r}$ with respect to generating set $1,g_1,\ld, g_s$, then
$\det \Lambda\cdot \det 
\Lambda^1_1$ defines a free divisor in $\CC^{r+1}$.
  \end{proof}

\begin{rem} With the hypotheses of the main theorem, the ideal 
\beq \bigl(\det f^*\Lambda^1_1,\ \det f^*\Lambda^2_2, \ld,
\ \det f^*\Lambda^{k-1}_{k-1}\bigr)\eeq
gives $D^k_1(f)$ the structure of a (possibly non-reduced) 
complete intersection.
\end{rem}

\begin{rem} It is slightly unexpected 
that, as the main theorem shows, for a corank 1 germ there is a presentation of
$\OO_{D^{k+1}(f)}$ over $\OO_{D^k(f)}$ whose entries are the pull-backs to 
$D^k(f)$ of functions
on the target of $f$. For the simplest map-germ of corank 2,
$$f(a,b,c,d,x,y)=(a,b,c,d,x^2+ay,xy+bx+cy,y^2+dx)$$
this cannot be the case even for a presentation of $\OO_{D^2(f)}$ over $\OO_{D^1(f)}$.
For, as an easy calculation shows, the Fitting ideal $\textnormal{Fitt}_3((\pi^2_1)_*\OO_{D^2(f)})$ is contained
in $\mathfrak{m}_{\CC^6,0}$ but not contained in $f^*(\mathfrak{m}_{\CC^7,0})$.
\end{rem}

%%%%%%%%%%%%%%%%%%%%%%%%%%%%%%%%%%%%%%%%%%%%%%

\section{A different interpretation of \texorpdfstring{$D^2_1(f)$}{D21(f)}}\label{sect-diff}

In this section, we identify $\Lambda^1_1$ as the matrix of a presentation of the 
kernel of the multiplication map 
$\mu \colon \mathcal{O}_{\mathbb{C}^{n},0} 
\otimes_{\mathcal{O}_{\mathbb{C}^{n+1}},0}\mathcal{O}_{\mathbb{C}^n,0}  \rightarrow 
\mathcal{O}_{\mathbb{C}^n,0}$, for a finite map-germ $f\in \mathcal{E}_{n,n+1}^0$. 
Suppose that $Q(f)=\mathbb{C}\cdot \{1, g_1,\ldots, g_r\}$. Then 
$\mathcal{O}_{\mathbb{C}^n,0}=\mathcal{O}_{\mathbb{C}^{n+1},0}\cdot 
\{1,g_1,\ldots,g_r\}$. Moreover, $\textnormal{Ker}(\mu)=\mathcal{O}_{\mathbb{C}^n,0}
\cdot \{g_1\otimes 1 -1\otimes g_1, \ldots, g_r\otimes 1 - 1\otimes g_r\}$.

\begin{prop}\label{pkerm} Let $f\in\mathcal{E}_{n,n+1}^0$ be finite and generically one-to-one. Suppose that
\beq \label{sequence} \xymatrix{0\ar[r]&\OO_{\CC^{n+1},0}^{r+1}\ar[r]^\Lambda&\OO_{\CC^{n+1},0}^{r+1}\ar[r]^G&\OO_{\CC^n,0}\ar[r]&0} \eeq is a presentation in which $\Lambda$ is symmetric and the first element of the generating set $G$ is 
equal to 1. 
Then the following sequence is exact:
\begin{equation*}%\label{kerms} 
\xymatrix{
0 \ar[r] &  \mathcal{O}_{\mathbb{C}^{n},0}^{r} \ar[r]^{f^*\Lambda^1_1} &  \mathcal{O}_{\mathbb{C}^{n},0}^{r} \ar[r]^{\Delta G} & \textnormal{Ker}(\mu) \ar[r] & 0} \end{equation*} where $\Delta G$ is defined by
$\hat{e}_i \mapsto g_{i}\otimes 1-1\otimes g_{i}$ for $ i=1,\ldots, r.$ \end{prop}

\begin{proof} We tensor (\ref{sequence}) on the right by 
$\mathcal{O}_{\mathbb{C}^n,0}$ over $\mathcal{O}_{\mathbb{C}^{n+1},0}$ to 
get \begin{equation*}  \mathcal{O}_{\mathbb{C}^{n},0}^{r+1}\xrightarrow{f^*\Lambda}  
\mathcal{O}_{\mathbb{C}^{n},0}^{r+1} \xrightarrow{G'} \mathcal{O}_{\mathbb{C}^{n},0} 
\otimes_{\mathcal{O}_{\mathbb{C}^{n+1},0}}\mathcal{O}_{\mathbb{C}^n,0}
\rightarrow 0 \end{equation*} where $G'\colon e_i\mapsto g_{i-1}\otimes 1$. By the 
exactness of (\ref{sequence}) and the fact that  $\Lambda$ is symmetric, 
$0=\left (G\cdot \Lambda\right )^t=\Lambda \cdot G^t$. So,  $G^t \colon 
\mathcal{O}_{\mathbb{C}^n,0}\rightarrow  \mathcal{O}_{\mathbb{C}^n,0}^{r+1}$ 
is a map into the kernel of $f^*\Lambda$. We claim that $\textnormal{im}(G^t)
=\textnormal{Ker}(f^*\Lambda)$. To see this, suppose that $B=
\begin{bmatrix}  b_0 &  b_1 & \ldots & b_r \end{bmatrix}^t$ 
lies in $\text{Ker}\,f^*\Lambda$, 
for some $b_0,\ldots,b_r\in \mathcal{O}_{\mathbb{C}^n,0}$. Then
\begin{equation*} 0=f^*\Lambda \cdot \big(\left[\begin{array}{c}  b_0 \\  b_1 \\ 
\ldots \\ b_r \end{array} \right]-b_0\left[\begin{array}{c}  1 \\  g_1 \\ 
\vdots \\  g_r \end{array} \right]\big)=f^*\Lambda\cdot \left[\begin{array}{c} 
0 \\  b_1-b_0g_1 \\ \vdots \\  b_r-b_0g_r \end{array} \right]. \end{equation*} 
and in particular $f^*(\Lambda^1_1)(b_1-b_0g_1,\ld,b_2-b_0g_r)^t=0$. However, 
$\det\,f^*\Lambda^1_1$ generates $f^*(\text{Fitt}_1(f_*\OO_{\CC^n,0}))$  
and is thus a defining equation of $D^2_1(f)$ (see e.g. \cite[Theorem 3.4 and 
Lemma 3.3]{mond-pellikaan}). Since $D^2_1(f)$  has codimension 1 in $D^1(f)$, 
$\det\,f^*\Lambda_1^1$ is not identically zero, and this forces $b_i=b_0g_i$ for 
$i=1,\ld,r$. 

Obviously, $G^t$ is one-to-one. Thus, we have the exact sequence
\begin{equation*}\label{exactfl}\xymatrix{ 0 \ar[r] &  \mathcal{O}_{\mathbb{C}^n,0} \ar[r]^{G^t} &
\mathcal{O}_{\mathbb{C}^n,0}^{r+1} \ar[r]^{f^*\Lambda} & \mathcal{O}_{\mathbb{C}^n,0}^{r+1} \ar[r]^-{G'} &
\mathcal{O}_{\mathbb{C}^{n},0}\otimes _{\mathcal{O}_{\mathbb{C}^{n+1},0} } \mathcal{O}_{\mathbb{C}^n,0} \ar[r] & 0 . } \end{equation*}
The fact that $f^*\lambda\cdot G^t=0$ means that the first column of $f^*\lambda$ is a linear combination of the remaining columns (recall that the first member of the set of generators, and thus the first entry in the column matrix $G^t$, is equal to $1$). Thus we obtain the exact sequence
\begin{equation}\label{les}\xymatrix{0 \ar[r] &  \mathcal{O}_{\mathbb{C}^n,0} \ar[r]^{\beta} &
\mathcal{O}_{\mathbb{C}^n,0}^{r+1} \ar[r]^{\alpha} & \mathcal{O}_{\mathbb{C}^n,0}^{r+1} \ar[r]^-{G''} &
\mathcal{O}_{\mathbb{C}^{n},0}\otimes _{\mathcal{O}_{\mathbb{C}^{n+1},0} } \mathcal{O}_{\mathbb{C}^n,0} \ar[r] & 0  } \end{equation} with
\begin{gather*}
\alpha=\left[\begin{array}{c|ccc} 0 & 0 & \ldots & 0 \\ \hline 0 & & \\ \vdots & & f^*\Lambda^1_1\\ 0 & & \end{array}\right], \hskip5pt \beta=\begin{bmatrix} 1 \\ 0 \\ \vdots \\ 0 \end{bmatrix}, \hskip5pt
G''\colon e_i \mapsto \begin{cases} 1\otimes 1 & \text{if $i=0$}, \\ g_{i}\otimes 1-1\otimes g_{i} & \text{if $i=1,\ldots, r$} \end{cases}.\end{gather*}
We minimalise \eqref{les} to the exact sequence
\begin{equation}\label{Kbullet}  0 \rightarrow \mathcal{O}_{\mathbb{C}^n,0}^{r}  \xrightarrow{\alpha_1}  \mathcal{O}_{\mathbb{C}^n,0}^{r+1} \xrightarrow{G''}
\mathcal{O}_{\mathbb{C}^{n},0}\otimes _{\mathcal{O}_{\mathbb{C}^{n+1},0} } \mathcal{O}_{\mathbb{C}^n,0} \rightarrow 0\end{equation} where $\alpha_1$ is the matrix obtained from $\alpha$ by deleting the first column. Since $\textnormal{Ker}(\mu)$ is generated by $\{g_{i}\otimes 1-1\otimes g_{i}\mid i=1,\ldots,r\}$,
$\mathcal{O}_{\mathbb{C}^{n},0} \otimes_{\mathcal{O}_{\mathbb{C}^{n+1},0}}\mathcal{O}_{\mathbb{C}^n,0}\cong\mathcal{O}_{\mathbb{C}^{n},0}\oplus \textnormal{Ker}(\mu)$. Hence, $f^*\Lambda_1^1$ gives a presentation of $\textnormal{Ker}(\mu)$. \end{proof}

%%%%%%%%%%%%%%%%%
\appendix
\section{Lifting of presentations}\label{app}

Consider the set
\begin{equation*}
\ID^k(f):=\{(x_1,\ldots,x_k)\in (\mathbb{C}^n)^k\mid f(x_1)=\cdots=f(x_k)\}
\end{equation*} which was referred to as  \textit{idiot's multiple point space} by Houston (\cite{houston-spectral}). Notice that $D^k(f)$ and $\ID^k(f)$ agree outside the diagonal of $(\mathbb{C}^n)^k$, and $D^k(f)\subset \ID^k(f)$.
In particular, $\ID^2(f)$ is the fibre product $(\mathbb{C}^n \times_{\mathbb{C}^{n+1}} \mathbb{C}^n,0)$ defined by the commutative diagram
\begin{equation*}\xymatrix{ (\mathbb{C}^{n},0) \ar[r]^f & (\mathbb{C}^{n+1},0) \\ (\mathbb{C}^n\times_{\mathbb{C}^{n+1}} \mathbb{C}^n,0) \ar[u]^{pr_1} \ar[r]^>>>>>{pr_2} & (\mathbb{C}^{n},0) \ar[u]_{f} }\end{equation*} where $pr_1$ (resp. $pr_2$) is the projection to the first (resp. the second) factor. The module $\OO_{\ID^2(f)}\cong \OO_{\mathbb{C}^{n},0}\otimes _{\OO_{\mathbb{C}^{n+1},0}} \OO_{\mathbb{C}^n,0} $ has two $\mathcal{O}_{\mathbb{C}^n,0}$-module structures induced from $pr_1$ and $pr_2$.
A resolution of $\OO_{\ID^2(f)}$ is given by (\ref{Kbullet}) and 
$\mathcal{O}_{\ID^2(f)}\cong \mathcal{O}_{\mathbb{C}^n,0}\oplus \textnormal{Ker}(\mu)$.

\begin{rem} If $f$ has corank 1, $\textnormal{Ker}(\mu)\cong \mathcal{O}_{D^2(f)}$  by Proposition \ref{pkerm} and the main theorem (cf. \cite[Prop. 3.2]{klu-curv}). We can form an isomorphism explicitly: define
\begin{eqnarray*} T\colon \textnormal{Ker}(\mu)  \rightarrow &(y_2-y_1)\mathcal{O}_{D^2(f)}  & \rightarrow \mathcal{O}_{D^2(f)}  \\  y\otimes 1-1\otimes y \mapsto &  y_2-y_1  &  \\  &  h & \mapsto h(y_2-y_1)^{-1}
\end{eqnarray*} and its inverse
\begin{eqnarray*} T^{-1}\colon   \mathcal{O}_{D^2(f)}  \rightarrow &(y_2-y_1)\mathcal{O}_{D^2(f)}  & \rightarrow  \textnormal{Ker}(\mu)  \\  h \mapsto  &  (y_2-y_1)h  &  \\  &  y_2-y_1 & \mapsto y\otimes 1-1\otimes y .
\end{eqnarray*}
So, $\mathcal{O}_{\ID^2}\rightarrow \mathcal{O}_{D^2}$ is the composition of
$\mathcal{O}_{\mathbb{C}^{n},0}\oplus \textnormal{Ker}(\mu)\rightarrow  \textnormal{Ker}(\mu)$ and  $T$. \end{rem}

 If $f$ has corank $\geq 2$ there is no isomorphism between $\textnormal{Ker}(\mu)$ and $\mathcal{O}_{D^2(f)}$. However, we can still compare the resolutions of $\mathcal{O}_{\ID^2}$ and  $\mathcal{O}_{D^2}$.

\begin{rem} Let $f\in\mathcal{E}_{n,n+1}^0$ be of corank $\geq 2$ with 
$Q(f)=\mathbb{C}\cdot \{1,g_1,\ldots, g_r\}$. Suppose that $q(f)=q(\pi^2_1)$. 
Let $\Lambda$ be the matrix of a presentation of $f_*\OO_{\CC^n}$ with respect to 
generators $1,g_1,\ld, g_r$, and let
$\Lambda_1$ denote $\Lambda$ with its first column removed.
We may take 
$Q(\pi^2_1)=\mathbb{C}\cdot \{1, g_1(\textbf{x}_2),\ldots, g_r(\textbf{x}_2)\}$. Then
\begin{equation*}\xymatrixcolsep{4pc}\xymatrix{ 0\ar[r] & \mathcal{O}_{\mathbb{C}^{n},
0}^r \ar[r]^{ f^*\Lambda_1} \ar[d]^B & \mathcal{O}_{\mathbb{C}^{n},0}^{r+1} 
\ar[r]^>>>>>>>{G'} \ar@{=}[d] & \mathcal{O}_{\ID^2}\ar[r] \ar[d]^{\hat{T}} & 0\\
0\ar[r] & \mathcal{O}_{\mathbb{C}^{n},0}^{r+1} \ar[r] & \mathcal{O}_{\mathbb{C}^{n},0}^{r+1} \ar[r]^<<<<<<<<<<{M} & \mathcal{O}_{D^2(f)} \ar[r] & 0}\end{equation*} is a commutative diagram with
$\hat{T}:g_i\otimes 1\mapsto g_i(\textbf{x}_2)$, $M=\begin{bmatrix} 1 & g_1(\textbf{x}_2) & \cdots & g_r(\textbf{x}_2) \end{bmatrix}$,
$G'=\begin{bmatrix} 1\otimes 1 & g_1\otimes 1-1\otimes g_1 & \cdots & g_r\otimes 1-1\otimes g_r\end{bmatrix}$ and $B$ the lift of the identity.
\end{rem}

\begin{rem}  Let $f\in\mathcal{E}_{n,n+1}^0$ be of corank $\geq 2$ with 
$Q(f)=\mathbb{C}\cdot \{1,g_1,\ldots, g_r\}$. Suppose that $q(f)\geq q(\pi^2_1)$. 
We may assume that $Q(\pi^2_1)=\mathbb{C}\cdot \{1, g_1(\textbf{x}_2),\ldots, 
g_{s}(\textbf{x}_2)\}$ 
($s\leq r$). Then there exist $a_{ij}\in \mathcal{O}_{\mathbb{C}^{n},0}$, 
$i,j=0,\ldots,s$, satisfying
$$g_{s+i}(\textbf{x}_2)=a_{i0}(\textbf{x}_1)+\sum_{j=1}^{s}a_{ij}(\textbf{x}_1)\cdot g_j(\textbf{x}_2)$$  for all $i=0,\ldots,r-1$. So,
\begin{equation*}\xymatrixcolsep{4pc}\xymatrix{ 0\ar[r] & \mathcal{O}_{
\mathbb{C}^{n},0}^r \ar[r]^{ f^*\Lambda_1} \ar[d]^B & \mathcal{O}_{\mathbb{C}^{n},
0}^{r+1} \ar[r]^>>>>>>>{G'} \ar[d]^A & \mathcal{O}_{\ID^2} \ar[r] \ar[d]^{\hat{T}} 
& 0\\ 0\ar[r] & \mathcal{O}_{\mathbb{C}^{n},0}^{s} \ar[r] & \mathcal{O}_{\mathbb{C}^{n},
0}^{s} \ar[r]^<<<<<<<<<<{M} & \mathcal{O}_{D^2(f)} \ar[r] & 0}\end{equation*} 
is a commutative diagram with 
$\hat{T}\colon g_i\otimes 1\mapsto g_i(\textbf{x}_2)$,   $M=\begin{bmatrix} 1 & 
g_1(\textbf{x}_2) & \cdots & g_{s}(\textbf{x}_2) \end{bmatrix}$,  
$G'=\begin{bmatrix} 1\otimes 1 & g_1\otimes 1-1\otimes g_1 & \cdots & g_r\otimes 
1-1\otimes g_r\end{bmatrix}$,
$A=\left[\begin{array}{c|c} &  \\ \textnormal{Id}_{r\times r} & (a_{ij}) \\ 
& \end{array}\right]$ and $B$ the lift of  $A$.
\end{rem}

\bibliographystyle{amsplain}
\bibliography{a-reference}

\providecommand{\bysame}{\leavevmode\hbox to3em{\hrulefill}\thinspace}
\providecommand{\MR}{\relax\ifhmode\unskip\space\fi MR }
\providecommand{\MRhref}[2]{%
  \href{http://www.ams.org/mathscinet-getitem?mr=#1}{#2}
}
\providecommand{\href}[2]{#2}
\begin{thebibliography}{10}

\bibitem{altintas}
A.~Altintas, \emph{Multiple point spaces and finitely determined map-germs},
  Ph.D. thesis, University of Warwick, 2011.

\bibitem{eisenbud}
D.~Eisenbud, \emph{Commutative algebra with a view toward algebraic geometry},
  Graduate Text in Mathematics, vol. 150, Springer, New York, 1995.

\bibitem{gaffney-multiple}
T.~Gaffney, \emph{Multiple points and associated loci}, Singularities, Part 1,
  Proc. Sympos. Pure Math., no.~40, American Mathematical Society, RI, 1983,
  pp.~429--437.

\bibitem{goryunov}
V.~Goryunov, \emph{Semi-simplicial resolutions and homology of images and
  discriminants of mappings}, Proc. London Math. Soc. \textbf{70} (1995),
  no.~3, 363--385.

\bibitem{goryunov-mond}
V.~Goryunov and D.~Mond, \emph{Vanishing cohomology of singularities of
  mappings}, Compositio Mathematica \textbf{89} (1993), 45--80.

\bibitem{greuel-l-s}
G.-M. Greuel, C.~Lossen, and E.~Shustin, \emph{Introduction to singularities
  and deformations}, Springer-Verlag, Berlin, 2007.

\bibitem{houston-spectral}
K.~Houston, \emph{An introduction to the image computing spectral sequence},
  Singularity Theory (Liverpool 1996), London Maths. Soc. Lecture Notes Series,
  vol. 263, Cambridge University Press, 1999.

\bibitem{houston09}
\bysame, \emph{Stratification of unfoldings of corank 1 singularities}, Quart.
  J. Math. \textbf{61} (2010), 413--435.

\bibitem{klu-curv}
S.~Kleiman, J.~Lipman, and B.~Ulrich, \emph{The multiple-point schemes of a
  finite curvilinear map of codimension one}, Ark. Mat. \textbf{34} (1996),
  285--326.

\bibitem{kleiman-ulrich}
S.~Kleiman and B.~Ulrich, \emph{Gorenstein algebras, symmetric matrices,
  self-linked ideals, and symbolic powers}, Transactions of the American
  Mathematical Society \textbf{349} (1997), no.~12, 4973--5000.

\bibitem{marar}
W.~L. Marar, \emph{Mapping fibrations and multiple point schemes}, Ph.D.
  thesis, University of Warwick, 1989.

\bibitem{marar-mond}
W.~L. Marar and D.~Mond, \emph{Multiple point schemes for corank 1 maps}, J.
  London Math. Soc. \textbf{39} (1989), no.~2, 553--567.

\bibitem{matherVI}
J.~Mather, \emph{Stability of $\textnormal{C}^\infty$-mappings
  $\textnormal{VI}$: {T}he nice dimensions}, Proc. Liverpool Singularities
  Symposium Vol. 1 (C.~T.~C. Wall, ed.), Lecture Notes in Mathematics, vol.
  192, Springer, Berlin, 1970, pp.~207--253.

\bibitem{mond87}
D.~Mond, \emph{Some remarks on the geometry and classification of germs of maps
  from surfaces to 3-spaces}, Topology \textbf{26} (1987), 361--383.

\bibitem{mond-montaldi}
D.~Mond and J.~Montaldi, \emph{Deformations of maps on complete intersections,
  {D}amon's $\mathcal{K}_{V}$-equivalence and bifurcations}, Singularities
  (Lille 1991), London Maths. Soc. Lecture Notes Series, vol. 201, Cambridge
  University Press, 1994.

\bibitem{mond-pellikaan}
D.~Mond and R.~Pellikaan, \emph{Fitting ideals and multiple points of analytic
  mappings}, Algebraic Geometry and Complex Analysis (P\'{a}tzcuaro, 1987),
  Lecture Notes in Mathematics, vol. 1414, Springer, Berlin, 1989.

\bibitem{mond-schulze}
D.~Mond and M.~Schulze, \emph{Adjoint divisors and free divisors},
  arXiv:1001.1095v3, 2010.

\bibitem{teissier}
B.~Teissier, \emph{The hunting of invariants in the geometry of discriminants},
  Nordic Summer School/NAVF, Symposium in Mathematics, Oslo, August 5-25, 1976,
  pp.~565--677.

\end{thebibliography}

\end{document}